\documentclass[11pt]{amsart}
\usepackage[left=2.65cm,right=2.65cm,top=2.6cm,bottom=2.6cm]{geometry}
\usepackage{amsmath,amsfonts,amsthm,amsopn,amssymb}
\usepackage[hyperpageref]{backref}
\usepackage{cite,marginnote}
\usepackage{amsmath}
\usepackage{color,enumitem,graphicx}
\usepackage[colorlinks=true,urlcolor=blue,
citecolor=red,linkcolor=blue,linktocpage,pdfpagelabels,
bookmarksnumbered,bookmarksopen]{hyperref}
\usepackage[english]{babel}
\numberwithin{equation}{section}

\makeindex
\newtheorem{theorem}{Theorem}[section]
\newtheorem{proposition}[theorem]{Proposition}
\newtheorem{lemma}[theorem]{Lemma}
\newtheorem{remark}[theorem]{Remark}
\newtheorem{example}[theorem]{Example}
\newtheorem{corollary}[theorem]{Corollary}
\newtheorem{definition}[theorem]{Definition}

\newcommand{\ud}{\mathrm{d}}
\newcommand{\RN}{\mathbb R^N}

\newcommand{\bt}{\begin{theorem}}
\newcommand{\et}{\end{theorem}}
\newcommand{\bl}{\begin{lemma}}
\newcommand{\el}{\end{lemma}}
\newcommand{\bd}{\begin{definition}}
\newcommand{\ed}{\end{definition}}
\newcommand{\bc}{\begin{corollary}}
\newcommand{\ec}{\end{corollary}}
\newcommand{\bp}{\begin{proof}}
\newcommand{\ep}{\end{proof}}
\newcommand{\bx}{\begin{example}}
\newcommand{\ex}{\end{example}}
\newcommand{\bi}{\begin{exercise}}
\newcommand{\ei}{\end{exercise}}
\newcommand{\bo}{\begin{proposition}}
\newcommand{\eo}{\end{proposition}}
\newcommand{\br}{\begin{remark}}
\newcommand{\er}{\end{remark}}
\newcommand{\be}{\begin{equation}}
\newcommand{\ee}{\end{equation}}
\newcommand{\ba}{\begin{align}}
\newcommand{\ea}{\end{align}}
\newcommand{\bn}{\begin{enumerate}}
\newcommand{\en}{\end{enumerate}}
\newcommand{\bg}{\begin{align*}}
\newcommand{\eg}{\end{align*}}
\newcommand{\bcs}{\begin{cases}}
\newcommand{\ecs}{\end{cases}}
\newcommand{\bean}{\begin{eqnarray*}}
\newcommand{\eean}{\end{eqnarray*}}

\newcommand{\iy}{\infty}

\newcommand{\s}{\section}

\newcommand{\pa}{\partial}

\newcommand{\R}{\mathbb R}

\newcommand{\rg}{\rightarrow}

\newcommand{\lab}{\label}

\renewcommand{\epsilon}{\varepsilon}

\title[Multiplicity of normalized solutions for fractional NLS]{ Multiplicity of normalized solutions for the \\ fractional Schr\"{o}dinger equation with potentials}

\author[X. Zhang]{Xue Zhang}
\author[M. Squassina]{Marco Squassina}
\author[J. J.\ Zhang]{Jianjun Zhang}

\address[X. Zhang]{\newline\indent College of Mathematica and Statistics
\newline\indent
Chongqing Jiaotong University
\newline\indent
Chongqing 400074, China}
\email{zhangxue@mails.cqjtu.edu.cn}

\address[Marco Squassina]{\newline\indent Dipartimento di Matematica e Fisica
	\newline\indent
	Universit\`a Cattolica del Sacro Cuore
	\newline\indent
	Via dei Musei 41, Brescia, Italy}
\email{marco.squassina@unicatt.it}

\address[J. J.\ Zhang]{\newline\indent College of Mathematica and Statistics
\newline\indent
Chongqing Jiaotong University
\newline\indent
Chongqing 400074, China}
\email{zhangjianjun09@tsinghua.org.cn}

\thanks{Marco Squassina is supported by Gruppo Nazionale per l'Analisi Ma\-te\-ma\-ti\-ca, la Probabilit\`a e le loro Applicazioni, while
Xue Zhang and Jianjun Zhang are supported by Joint Training Base Construction Project for Graduate Students in Chongqing (JDLHPYJD2021016).}

\date{}

\begin{document}

\begin{abstract}
We are concerned with the existence and multiplicity of normalized solutions to the fractional Schr\"{o}dinger equation
\begin{eqnarray*} \left\{
	\begin{array}{ll}
		(-\Delta)^su+V(\epsilon x)u=\lambda u+h(\epsilon x)f(u)&\mbox{in}\ \mathbb{R}^N,\\
		\displaystyle\int_{\mathbb{R}^N}|u|^2dx=a,
	\end{array}
	\right.
\end{eqnarray*}
where $(-\Delta)^s$ is the fractional Laplacian, $s\in(0,1)$, $a,\epsilon>0$, $\lambda\in\R$ is an unknown parameter that appears as a Lagrange multiplier,
$h:\mathbb{R}^N\rightarrow[0,+\infty)$ are bounded and continuous, and $f$ is $L^2$-subcritical. Under some assumptions on the potential $V$, we show that
the existence of normalized solutions depends on global maximum points of $h$ when $\epsilon$ is small enough.
\end{abstract}

\maketitle

\s{Introduction}
\renewcommand{\theequation}{1.\arabic{equation}}
\subsection{Background and motivation}
In this paper, we investigate the multiplicity of normalized solutions for the  fractional Schr\"{o}dinger equation
\begin{equation}\lab{11}
	i\frac{\partial \psi}{\partial t}=(-\Delta)^s\psi+V(x)\psi -g(|\psi|^2)\psi\ \ \rm in\  \R^N,
\end{equation}
where $0<s<1$, $ i$ denotes the imaginary unit and $\psi(x,t)$ is a complex wave. A solution of \eqref{11} is called a standing wave solution if it has the form $\psi(x,t)=e^{-i\lambda t}u(x)$ for some $\lambda\in\R$. $(-\Delta)^s$ stands for the fractional Laplacian and if $u$ is small enough, it can be computed by the following singular integral
$$
(-\Delta)^su=C(N,s){\rm{P.V.}} \int_{\RN}\frac{u(x)-u(y)}{|x-y|^{N+2s}}\ud y.
$$
Here the symbol $\rm{P.V.}$ is the Cauchy principal value and $C(N,s)$ is a suitable positive normalizing constant.

The operator $(-\Delta)^s$ can be seen as the infinitesimal generators of L\'{e}vy stable diffusion processes \cite{Appleb}, it originates from describing various phenomena in the field of applied science, such as fractional quantum mechanics, barrier problem, markov processes and phase transition phenomenon, see \cite{Guan,Las,Sil,Sir}. In recent decades, the study of problems of fractional Schr\"{o}dinger equation has attracted wide attention, see e.g.
\cite{Serv,Shang,Yan} and references therein.

In \cite{Alves}, Alves considered the following class of elliptic problems with a $L^2$-subcritical nonlinear term
\begin{eqnarray}\label{alves}
\left\{
\begin{array}{ll}
	-\Delta u=\lambda u+h(\epsilon x)f(u)&\mbox{in}\ \R^N,\\
	\displaystyle\int_{\mathbb{R}^N}|u|^2\mathrm{d}x=a.
\end{array}
\right.
\end{eqnarray}
By using the variational approaches, the author shows that problem \eqref{alves} admits multiple normalized solutions if $\epsilon$ is small enough. Particularly, the numbers of normalized solutions are at least the numbers of global maximum points of $h$. Moreover, for the following class of problem
\begin{eqnarray*}
\left\{
\begin{array}{ll}
	-\Delta u+V(\epsilon x)u=\lambda u+f(u)&\mbox{in}\ \R^N,\\
	\displaystyle\int_{\mathbb{R}^N}|u|^2\mathrm{d}x=a,
\end{array}
\right.
\end{eqnarray*}
a similar result is also obtained for some negative and continuous potential $V$.

Motivated by \cite{Alves}, our interest is mainly focused on the fractional case with both potentials and weights. Actually, our purpose of this paper is devoted to the multiplicity of normalized solutions for the fractional Schr\"{o}dinger equation
\begin{eqnarray}\label{1.2}
\left\{
\begin{array}{ll}
	(-\Delta)^su+V(\epsilon x)u=\lambda u+h(\epsilon x)f(u)&\mbox{in}\ \R^N,\\
	\displaystyle\int_{\mathbb{R}^N}|u|^2\mathrm{d}x=a,&
\end{array}
\right.
\end{eqnarray}
where $s\in(0,1)$, $a,\epsilon>0$, $\lambda\in\R$ is an unknown parameter that appears as a Lagrange multiplier.

In the local case, when $s=1$, the fractional laplace $(-\Delta)^s$ reduces to the local differential opterator $-\Delta$. If $V(x)\equiv0$, Jeanjean's \cite{Jean} exploited the mountain pass geometry to deal with existence of normalized solutions in purely $L^2$-supercritical, we refer \cite{Bonhe,Guo,Hir,Li} for more results in this type of problems. In
\cite{Miao}, they considered the related problem for $q=2+\frac{4}{N}$.
The multiplicity of normalized solutions for the Schr\"{o}dinger equation or systems has also been extensively investigated, see \cite{Gou,Jean,Jean,Shi}.

For the non-potential case, a large body of literature is devoted to the following problem:
\begin{eqnarray}
	\left\{
	\begin{array}{ll}
		-\Delta u=\lambda u+g(u)&\mbox{in}\ \R^N,\\
		\displaystyle\int_{\mathbb{R}^N}|u|^2\mathrm{d}x=a^2.&
	\end{array}
	\right.
\end{eqnarray}
In particular, for the case $g(u)=|u|^{p-1}u$, by assuming $H^1$-precompactness of any minimizing sequences, Cazenave and Lions \cite{Cazenave} showed the attainability of the $L^2$-constraint minimization problem and orbital stability of global minimizers, it is assumed that $E_\alpha<0$ for all $\alpha>0$, and then, the strict subadditivity condition:
\begin{equation}\lab{E}
E_{\alpha+\beta}<E_\alpha+E_\beta
\end{equation}
holds. However, when dealing with the general function $g$, it is difficult to show \eqref{E} holds. Shibata \cite{Shi} proved the subadditivity condition \eqref{E} using a scaling argument.

In addition, if $V(x)\not\equiv0$, Ikoma and Miyamoto \cite{Ikoma}
studies the existence and nonexistence  of a minimizer of the $L^2$-constraint minimization problem
$$e(a)=\mbox{inf}\{E(u)|u\in H^1(\RN),|u|^2_2=a\},
$$
where
$$
E(u)=\frac{1}{2}\int_{\R^N}(|\nabla u|^{2}\mathrm{d}x+V(x)|u|^2)\mathrm{d}x-\int_{\R^N}F(u)\mathrm{d}x,
$$
$V$ and $f$ satisfy some suitable assumptions. They performed a careful analysis to exclude dichotomy and proved the precompactness of the modified minimizing sequence. When dealing with general nonlinear terms in mass subcritical cases, one can apply the subadditive inequality to prove the compactness of the minimzing sequence.

Zhong and Zou in \cite{Zhon} studied the existence of ground state normalized solution to Schr\"{o}dinger equations with potential under  different assumptions, and presented a new approach to establish the strict sub-additive inequality. Alves and Thin\cite{AlvesT} study the existence of multiple normalized solutions to the following class of ellptic problems
\begin{eqnarray}
	\left\{
	\begin{array}{ll}
		-\Delta u+V(\epsilon x)u=\lambda u+f(u)&\mbox{in}\ \R^N,\\
		\displaystyle\int_{\mathbb{R}^N}|u|^2\mathrm{d}x=a,&
	\end{array}
	\right.
\end{eqnarray}
where $\epsilon>0$, $V:\R^N\rightarrow[0,\infty)$ is a continuous  function, and $f$ is a differentiable function with $L^2$-subcritical growth. For normalized solutions of the nonlinear Schr\"{o}dinger equation with potential, we also see \cite{Bartsch,Ikoma2,Molle} and the references therein.

In the case $0<s<1$, few results are available. In the paper \cite{Yang} the author proved some existence and asymptotic results for the fractional nonlinear Schr\"{o}dinger equation. For the particular case of a combined nonlinearity of power type, namely $f(t)=\mu|u|^{q-2}u+|u|^{p-2}u$, $h(x)=1$ and $V(x)\equiv0$, $\rm{i.e}$ $2<q<p<2_s^*=\frac{2N}{N-2s}$. Dinh \cite{Dinh} studied the existence
and nonexistence of normalized solutions for the fractional Schr\"{o}dinger equations
\begin{align}\label{tt}
(-\Delta)^su+V(x)u=|u|^{p-2}u,\ \ \mbox{in}\ \ \RN.
\end{align}
By using the concentration-compactness principle, he showed  a complete classification
for the existence and non-existence of normalized solutions for the problem \eqref{tt}. For more results about the fractional Schr\"{o}dinger equations, we can refer to \cite{Luo,Feng} and the references therein.

\vskip4pt
\subsection{Main results} In what follows, we assume $f\in C^1(\RN,\R)$ is odd, continuous and satisfies the following assumptions on $f$.
\begin{itemize}
\item [($f_1$)] $\lim\limits_{t\rg0}\frac{|f(t)|}{|t|^{q-1}}=c>0$, where $2<q<\bar{p}=2+\frac{4s}{N}$.
\item [($f_2$)] $\lim\limits_{t\rg\infty}\frac{|f(t)|}{|t|^{p-1}}=0$, where $2<p<\bar{p}=2+\frac{4s}{N}$.
\item [($f_3$)] There exist $\alpha,\beta\in\R$ satisfying $2<\alpha\le \beta<\bar{p}$ such that $$0<\alpha F(t)\le tf(t)\le F(t)\beta \ \ \mbox{for any} \ t>0.$$
\end{itemize}
Moreover, $h$ and $V$ satisfy the following assumptions.
\begin{itemize}
\item [($A_1$)]  $h\in C(\RN,\R^+)$,
$0<h_{\iy}=\lim\limits_{|x|\rightarrow+\infty}h(x)<\max\limits_{x\in\RN}h(x)=h(a_i)$ for $1\le i\le k$ with $a_1=0$ and $a_j\not=a_i$ if $i\not=j$.
\item [($A_2$)]  $V\in C(\RN,\R)$, $V(a_i)=\inf\limits_{x\in\RN}V(x)<\lim\limits_{|x|\rightarrow+\infty}V(x)=0$ for $1\le i\le k$.
\end{itemize}
The problem \eqref{1.2} is variational and the associated energy functional is given by
\begin{equation}\lab{xx4}
I_\epsilon(u)=\frac{1}{2}\int_{\R^N}|(-\Delta)^{\frac{s}{2}}u|^{2}\mathrm{d}x+\frac{1}{2}\int_{\RN}V(\epsilon x)u^2\mathrm{d}x-\int_{\R^N}h(\epsilon x)F(u)\mathrm{d}x, \ u\in H^s(\RN)
\end{equation}
with
$$
\int_{\R^N}|(-\Delta)^{\frac{s}{2}}u|^{2}\mathrm{d}x=\iint_{\R^{2N}}\frac{|u(x)-u(y)|^2}{|x-y|^{N+2s}}\mathrm{d}x\mathrm{d}y.
$$
It is easy to know that $I_\epsilon\in C^1(H^s(\RN),\R)$ and $$I_\epsilon'(u)\varphi=\int_{\RN}(-\Delta)^{\frac{s}{2}}u(-\Delta)^{\frac{s}{2}}\varphi \mathrm{d}x+\int_{\RN}V(\epsilon x)u\varphi \mathrm{d}x-\int_{\RN}h(\epsilon x)f(u)\varphi \mathrm{d}x,\ \ \forall \varphi\in H^s(\RN).$$
The solutions to \eqref{1.2} can be characterized as critical points of the function  $I_\epsilon(u)$ constrained on the sphere
\begin{equation}\lab{q}
S_a=\left\{u\in H^s(\RN):\int_{\RN}|u|^2\mathrm{d}x=a\right\}
\end{equation}
Now, we are ready to state the main result of this paper.
\begin{theorem}\lab{Th1}
Suppose $(A_1),(A_2)$, $(f_1)-(f_3)$ hold, then there exists $\epsilon_1>0$ such that  problem \eqref{1.2} admits at least $k$ couples $(u_j,\lambda_j)\in H^s(\R^N)\times\R$ of weak solutions for $\epsilon\in(0,\epsilon_1)$ with $\int_{\R^N}|u_j|^2\mathrm{d}x=a$, $\lambda<0$ and $I_\epsilon(u_j)<0$ for $j=1,2,\cdots,k$.
\end{theorem}

The paper is organized as follows. In Section 2, we study the autonomous problem and give some useful results which will be used later. Section 3 is devoted to the non-autonomous problem. In Section 4, the proof of Theorem \ref{Th1} is given.

\vskip0.1in

\s{The autonomous problem}
\renewcommand{\theequation}{2.\arabic{equation}}

In this section, we focus on the existence of normalized solution for the autonomous problem
\begin{equation}\lab{21} \left\{
	\begin{array}{ll}
		(-\Delta)^su+\eta u=\lambda u+\mu f(u)&\mbox{in}\ \R^N,\\
		\displaystyle\int_{\mathbb{R}^N}|u|^2\mathrm{d}x=a,&
	\end{array}
	\right.
\end{equation}
where $s\in(0,1)$, $a,\mu>0$, $\eta\leq0$ and $\lambda\in\R$ is an unknown parameter that appears as a Lagrange multiplier. With the assumptions $(f_{1})-(f_{3})$, it is standard to show that the solutions to \eqref{21} can be characterized as critical points of the function as follows
\begin{equation}\lab{q2}
	J(u)=\frac{1}{2}\int_{\R^N}|(-\Delta)^{\frac{s}{2}}u|^{2}\mathrm{d}x+\frac{\eta}{2}\int_{\RN}u^2\mathrm{d}x-\mu\int_{\R^N}F(u)\mathrm{d}x
\end{equation}
restricted to the sphere $S_a$ given in \eqref{q}.
Meanwhile, set
$$
J_0(u)=\frac{1}{2}\int_{\R^N}|(-\Delta)^{\frac{s}{2}}u|^{2}\mathrm{d}x-\mu\int_{\R^N}F(u)\mathrm{d}x
$$
and
\begin{align*}
\Upsilon_a=\inf\limits_{S_a} J_0(u).
\end{align*}

\begin{theorem}\lab{Th2.1}
	Suppose that $f$ satisfies the conditions $(f_{1})-(f_{3})$. Then, problem \eqref{21} has a couple $(u,\lambda)$ solution, where $u$ is positive, radial and $\lambda<\eta$.
\end{theorem}
The proof of Theorem \ref{Th2.1} is standard. For the sake of convenience, we give the details. Before the proof, some lemmas are given below.
\begin{lemma}\lab{ee}
	Assume $u$ is a solution to \eqref{21}, then $u\in S_a\cap  P$, where
	$$
	P:=\left\{u\in H^s(\R^N)|\int_{\R^N}|(-\Delta)^{\frac{s}{2}}u|^2\mathrm{d}x+\frac{N\mu}{s}\int_{\RN}F(u)\mathrm{d}x-\frac{N\mu}{2s}\int_{\RN}f(u)u\mathrm{d}x=0\right\}.
	$$
\end{lemma}
\begin{proof}
	Let $u$ be a solution \eqref{21}, then we get
	\begin{equation}\lab{c}
		\int_{\R^N}|(-\Delta)^{\frac{s}{2}}u|^2\mathrm{d}x+(\eta-\lambda)\int_{\RN}u^2\mathrm{d}x-\mu\int_{\RN}f(u)u\mathrm{d}x=0,
	\end{equation}
In addition, one can show that $u$ satisfies the Pohozeav identity
	$$(N-2s)\int_{\R^N}|(-\Delta)^{\frac{s}{2}}u|^2\mathrm{d}x+N(\eta-\lambda)\int_{\RN}u^2-2N\mu\int_{\RN}F(u)=0.$$
	Combining with  \eqref{c}, we obtain that
	$$
	\int_{\R^N}|(-\Delta)^{\frac{s}{2}}u|^2\mathrm{d}x+\frac{N\mu}{s}\int_{\RN}F(u)-\frac{N\mu}{2s}\int_{\RN}f(u)u\mathrm{d}x=0.
	$$
\end{proof}
\begin{lemma}\lab{1}
	Assume $(f_1)-(f_2)$, then we have
	\begin{itemize}
		\item [$(i)$] $J$ is bounded from below on $S_a$,
		\item [$(ii)$] any minimizing sequence for $J$ is bounded in $H^s(\RN)$.
	\end{itemize}
\end{lemma}
\begin{proof}
	$(i)$ According the assumptions $(f_{1})-(f_{2})$, there exists $C>0$ such that
	\begin{equation}\lab{2333}
		|F(t)|\le C(|t|^q+|t|^p),\ \ \forall t\in\R.
	\end{equation}
	By the fractional Gagliardo-Nirenberg-Sobolev inequality \cite{Frank},
	\begin{equation}\lab{23}
		\int_{\R^N}|u|^\alpha\le C(s,N,\alpha)(\int_{\R^N}|(-\Delta)^{\frac{s}{2}}u|^2)^{\frac{N(\alpha-2)}{4s}}(\int_{\R^N}|u|^2)^{\frac{\alpha}{2}-\frac{N(\alpha-2)}{4s}},
	\end{equation}	
	for some positive constant $C(s,N,\alpha)>0$. Then, \eqref{2333} and \eqref{23} give that
\begin{align}\lab{233}
	J(u)\ge&\frac{1}{2}\int_{\R^N}(|(-\Delta)^{\frac{s}{2}}u|^2+\frac{\eta}{2}u^2)\mathrm{d}x-\frac{\mu C(s,N,q)}{q}a^{\frac{q}{2}-\frac{N(q-2)}{4s}}(\int_{\R^N}|(-\Delta)^{\frac{s}{2}}u|^2\mathrm{d}x)^\frac{N(q-2)}{4s}\nonumber\\
	&-\frac{\mu C(s,N,p)}{p}a^{\frac{p}{2}-\frac{N(p-2)}{4s}}(\int_{\R^N}|(-\Delta)^{\frac{s}{2}}u|^2\mathrm{d}x)^\frac{N(p-2)}{4s}.
\end{align}
Since $q,p\in(2,2+\frac{4s}{N})$, we infer that $0<\frac{N(q-2)}{4s},\frac{N(p-2)}{4s}<1$. Therefore $J(u)$ is bounded from below on $S_a$.
	
$(ii)$ Since $u\in S_a$, the conclusion immediately follows from \eqref{233}.
\end{proof}

The lemma above guarantees that $$E_{a}=\inf\limits_{u\in S_a}J(u)$$ is well defined. Now we study the properties of the function $J$ defined in \eqref{21} restrict to $S_a$ and prove Theorem \ref{Th2.1}.
\begin{lemma}\lab{2}
	For any $a>0$ and $\eta\leq0$, there holds $E_{a}<0$. In particular, we have $E_a<\frac{\eta a}{2}$.
\end{lemma}
\begin{proof}
	According $(f_1)$, $\lim\limits_{t\to0}\frac{qF(t)}{t^q}=c>0$ and then there exists $\zeta>0$ such that
	\begin{equation}\lab{25}
		\frac{qF(t)}{t^q}\ge\frac{c}{2},\ \forall t\in[0,\zeta].
	\end{equation}
	In fact, taking $u\in S_a\cap L^{\iy}(\RN)$ as a fixed nonnegative function, we define
$$
(\tau\ast u)(x)=e^{\frac{N}{2}\tau}u(e^\tau x),\ \ \mbox{for}\  \mbox{all}\  x\in\R^N\  \   \mbox{and}\ \mbox{all}\ \tau\in\R,
$$
then $\tau\ast u\in S_a$. Moreover, for $\tau<0$ and $|\tau|$ large enough, we have $$0\le e^{\frac{N}{2}\tau}u(x)\le\zeta,\ \  \forall x\in\R^N,$$
which combines with \eqref{25} give that$$\int_{\R^N}F(\tau\ast u)\mathrm{d}x\ge Ce^{\frac{(q-2)N\tau}{2}}\int_{\R^N}|u|^q\mathrm{d}x.$$
It follows that
	\begin{align}
		J(\tau\ast u)=&\frac{1}{2}\int_{\R^N}|(-\Delta)^{\frac{s}{2}}(\tau\ast u)|^{2}\mathrm{d}x+\frac{\eta a}{2}-
		\mu\int_{\R^N}F(\tau\ast u)\mathrm{d}x\nonumber\\
		\le&\frac{1}{2}e^{2s\tau}\int_{\R^N}|(-\Delta)^{\frac{s}{2}}u|^{2}\mathrm{d}x+\frac{\eta a}{2}-\mu Ce^{\frac{(q-2)N\tau}{2}}\int_{\R^N}|u|^q\mathrm{d}x.
	\end{align}
	Since $q\in(2,2+\frac{4s}{N})$, increasing $|\tau|$ if necessary, we have
	$$\frac{1}{2}e^{2s\tau}\int_{\R^N}|(-\Delta)^{\frac{s}{2}}u|^{2}\mathrm{d}x-\mu Ce^{\frac{(q-2)N\tau}{2}}\int_{\R^N}|u|^q\mathrm{d}x=K_\tau<0.$$
Hence, we obtain
	$$J(\tau\ast u)\leq K_{\tau}+\frac{\eta a}{2}<0$$
and then $E_{a}<0$. In particular, we have $E_a<\frac{\eta a}{2}$. The proof is complete.
\end{proof}
In the following, we adopt some idea introduced in \cite{Zhon} to get the sub-additive inequality.
\begin{lemma}\lab{3}
	For $\mu>0,\eta\le0$ and let $a,b>0$, then
	\begin{itemize}
		\item [$(i)$] $a\mapsto E_{a}$ is nonincreasing,
		\item [$(ii)$] $a\mapsto E_{a}$ is continuous,
		\item [$(iii)$] $E_{a+b}\le E_{a}+E_{b}$. If $E_a$ or $E_b$ can be attained, then $E_{a+b}<E_a+E_b$.
	\end{itemize}
\end{lemma}
\begin{proof}
	$(i)$. For any $\varepsilon>0$ small,  there exist $u\in S_{a}\cap C^{\infty}_0(\RN)$ and $v\in S_{b-a}\cap C^{\infty}_0(\RN)$ such that
	$$J(u)\leq E_{a}+\varepsilon,\ \ \ J_0(v)\le \Upsilon_{b-a}+\varepsilon.$$
    Since $u$ and $v$ have compact support, by using parallel translation, we can take $R$ large enough satisfying
	$$
	\tilde{v}(x)=v(x-R),\quad supp\ u \cap supp\ \tilde{v} =\emptyset.
	$$
	Then $u+\tilde{v}\in S_{b}$ and
	\begin{align*}
		E_{b}\le J(u+\tilde{v})=&\frac{1}{2}\iint_{\R^{2N}}\frac{|(u+\tilde{v})(x)-(u+\tilde{v})(y)|^2}{|x-y|^{N+2s}}\mathrm{d}x\mathrm{d}y+\frac{\eta}{2}|u+\tilde{v}|^2_2-\mu\int_{\RN}F(u+\tilde{v})\mathrm{d}x\\
		=&J(u)+J(\tilde{v})+\iint_{\R^{2N}}\frac{(u(x)-u(y))(\tilde{v}(x)-\tilde{v}(y))}{|x-y|^{N+2s}}\mathrm{d}x\mathrm{d}y,
	\end{align*}
	Suppose that
	$$supp\ u\subset B_R(0)\ \ \mbox{and}  \ \ supp\ \tilde{v}\subset B_{3R}(0)\backslash B_{2R}(0),$$
	we obtain
	\begin{align*}
		\iint_{\R^{2N}}\frac{(u(x)-u(y))(\tilde{v}(x)-\tilde{v}(y))}{|x-y|^{N+2s}}\mathrm{d}x\mathrm{d}y=&\iint_{\R^{2N}}\frac{u(x)\tilde{v}(x)-2u(x)\tilde{v}(y)+u(y)\tilde{v}(y)}{|x-y|^{N+2s}}\mathrm{d}x\mathrm{d}y\\
		=&\iint_{\R^{2N}}\frac{-2u(x)\tilde{v}(y)}{|x-y|^{N+2s}}\mathrm{d}x\mathrm{d}y,
	\end{align*}
	Noting that $|x-y|\geq R$ large enough, we have
	\begin{align}\label{dvvv}
		E_{b}\le J(u+\tilde{v})\le J(u)+J(\tilde{v})+\varepsilon\le J(u)+J_0(v)+\varepsilon\le E_{a}+\Upsilon_{b-a}+3\varepsilon\le E_a+3\varepsilon.
	\end{align}
 Here we used the fact $\Upsilon_{b-a}<0$. Then by \eqref{dvvv} and the arbitrariness of $\varepsilon$, we obtain that $E_{b}\le E_{a}$ for any $b>a>0$.
	
		$(ii)$. We prove the following two claims.
	\vskip3pt
	{\bf {Claim 1:}}  $\lim\limits_{h\rightarrow0^+}E_{a-h}\le E_{a}$.
	
	For $\varepsilon>0$, by the definition of $E_{a}$, there exists $u\in S_{a}$ such that
	\begin{equation}\label{33}
		E_{a}\le J(u)\le E_{a}+\varepsilon.
	\end{equation}
	Setting $$t=t(h)=(\frac{a-h}{a})^{\frac{1}{N}}$$
	and $u_t(x)=u(\frac{x}{t})$, we get
	\begin{equation}\label{333}
		\lim\limits_{h\rightarrow0^+}t=1 \ \ \mbox{and} \ \ |u_t|^2_2=t^Na=a-h.
	\end{equation}
	Then, by using $(i)$, we have $J(u_t)\ge E_{a-h}$. In addition,
	\begin{align*}
		J(u_t)=&\frac{t^{N-2s}}{2}\int_{\R^N}|(-\Delta)^{\frac{s}{2}}u|^{2}\mathrm{d}x+\frac{\eta t^N}{2}\int_{\RN}u^2\mathrm{d}x-\mu t^N\int_{\RN}F(u)\mathrm{d}x\\
		=&\frac{t^{N-2s}}{2}\int_{\R^N}|(-\Delta)^{\frac{s}{2}}u|^{2}\mathrm{d}x+t^N(J(u)-\frac{1}{2}\int_{\R^N}|(-\Delta)^{\frac{s}{2}}u|^{2}\mathrm{d}x)\\
		=&t^NJ(u)+\frac{t^{N-2s}(1-t^{2s})}{2}\int_{\R^N}|(-\Delta)^{\frac{s}{2}}u|^{2}\mathrm{d}x
	\end{align*}
	by \eqref{33} and \eqref{333}, we obtain
	$$\lim\limits_{h\rightarrow0^+}E_{a-h}\le E_{a}+\varepsilon.$$
Since $\varepsilon$ is arbitrary, the claim holds.
	\vskip0.1in
	 {\bf {Claim 2:}} $\lim\limits_{h\rightarrow0^+}E_{a+h}\ge E_{a}$,
	
	Actually, we consider the case $h=\frac{1}{n},n\in\mbox{N}$. Take $u_n\in S_{a+\frac{1}{n}}$ such that $J(u_n)\le E_{a+\frac{1}{n}}+\frac{1}{n}$. Set $$v_n(x):=\sqrt{\frac{na}{na+1}}u_n(x).$$
	By Lemma \ref{1}, we know $\{u_n\}$ is bounded in $H^s(\RN)$. Morever, we have
	$$|v_n|^2_2=\frac{na}{na+1}|u_n|^2_2=\frac{na}{na+1}(a+\frac{1}{n})=a.
	$$
	Hence, we get $u_n\in S_a$. On the other hand,
	$$||v_n-u_n||_{H^s(\RN)}=(1-\sqrt{\frac{na}{na+1}})||u_n||_{H^s(\RN)}\rightarrow 0\ \ \mbox{as}\ \ n\rightarrow+\infty,$$
	Then
	$$E_a\le\liminf\limits_{n\rightarrow+\infty}J(v_n)=\liminf\limits_{n\rightarrow+\infty}[J(u_n)+o_n(1)]=\lim\limits_{h\rightarrow0^+}E_{a+h}.$$
	Thus, we obtain that
	$$\lim\limits_{h\rightarrow0^+}E_{a+h}\ge E_{a}.$$
	Moreover, $E_{a-h}\ge E_{a}\ge E_{a+h}$ holds due to $(i)$. Hence, we get
	$$\lim\limits_{h\rightarrow0^+}E_{a-h}\ge E_a\ge \lim\limits_{h\rightarrow0^+}E_{a+h}.$$
	We complete the proof of $(ii)$.

	$(iii)$. Firstly, we prove that $$E_{\theta a}\le \theta E_{a} \ \mbox{for}\  \theta>1 \ \mbox{closing to}\  1.$$
	For any $\varepsilon>0$, we take $u\in S_a\cap P$ such that
	$$J(u)\le E_a+\varepsilon.$$
	Setting $\tilde{u}(x)=u(\nu^{-\frac{1}{N}}x)$ for $\nu\ge1$, by the assumption, we have $|\tilde{u}|^2_2=\nu a$ and
	\begin{align*}
		J(\tilde{u})=&\frac{1}{2}\int_{\R^N}|(-\Delta)^{\frac{s}{2}}\tilde{u}|^{2}\mathrm{d}x+\frac{\eta}{2}\int_{\RN}\tilde{u}^2\mathrm{d}x-\mu\int_{\R^N}F(\tilde{u})\mathrm{d}x\\
		=&\frac{1}{2}\nu^{\frac{N-2s}{N}}\int_{\R^N}|(-\Delta)^{\frac{s}{2}}u|^{2}\mathrm{d}x+\frac{\eta\nu}{2}\int_{\RN}u^2\mathrm{d}x-\mu\nu\int_{\R^N}F(u)\mathrm{d}x.
	\end{align*}
	Then, we get that
	\begin{align*}
		\frac{d}{d\nu}J(\tilde{u})=\frac{N-2s}{2N}\nu^{-\frac{2s}{N}}\int_{\R^N}|(-\Delta)^{\frac{s}{2}}u|^{2}\mathrm{d}x+\frac{\eta}{2}\int_{\RN}u^2\mathrm{d}x-\mu\int_{\RN}F(u)\mathrm{d}x.
	\end{align*}
	Since $u\in P$, we know
	$$\int_{\R^N}|(-\Delta)^{\frac{s}{2}}u|^2\mathrm{d}x+\frac{N\mu}{s}\int_{\RN}F(u)-\frac{N\mu}{2s}\int_{\RN}f(u)u\mathrm{d}x=0.$$
	Thus
	\begin{align*}
		\frac{d}{d\nu}J(\tilde{u})-J(u)=&(\frac{N-2s}{2N}\nu^{-\frac{2s}{N}}-\frac{1}{2})\int_{\R^N}|(-\Delta)^{\frac{s}{2}}u|^{2}\mathrm{d}x.\\
		=&(\frac{N-2s}{2N}\nu^{-\frac{2s}{N}}-\frac{1}{2})\frac{N\mu}{s}\int_{\RN}[\frac{1}{2}f(u)u-F(u)]\mathrm{d}x\\
		=&(\frac{N-2s}{2s}\mu\nu^{-\frac{2s}{N}}-\frac{N\mu}{2s})\int_{\RN}[\frac{1}{2}f(u)u-F(u)]\mathrm{d}x.
	\end{align*}
	Obviously, if $\xi>0$ small, it follows that
	\begin{equation}\lab{111}
		\frac{N-2s}{2s}\mu\nu^{-\frac{2s}{N}}-\frac{N\mu}{2s}<0,\ \mbox{for}\ \nu\in[1,1+\xi].
	\end{equation}
	Then by $\eqref{111}$ and $(f_3)$, we obtain that
	\begin{align*}
		\frac{d}{d\nu}J(\tilde{u})-J(u)\le(\frac{N-2s}{2s}\mu\nu^{-\frac{2s}{N}}-\frac{N\mu}{2s})(\frac{\alpha-2}{2})\int_{\RN}F(u)\mathrm{d}x<0,
	\end{align*}
	Namely,
	$$\frac{d}{d\nu}J(\tilde{u})-J(u)<0,\ \mbox{for }\  \forall\nu\in[1,1+\xi].$$
	Therefore, for any $\theta\in(1,1+\xi)$, we have
	$$J(\tilde{u})-J(u)=\int_{1}^{\theta}\frac{d}{d\nu}J(\tilde{u})d\nu<\int_{1}^{\theta}J(u)d\nu=J(u)(\theta-1).$$
Then,  it is easy to see that
	$$E_{\theta a}\le J(\tilde{u})\leq\theta J(u)\le\theta (E_{a}+\varepsilon),$$
Since the arbitrariness of $\varepsilon$, we get $$E_{\theta a}\le\theta E_{a},\ \theta\in(1,1+\xi).$$
and if $E_a$ is attained, we can take $u$ as a minimizer in the above step, then we have
	$$E_{\theta a}\le J(\tilde{u})<\theta J(u)=\theta E_{a},\ \ \theta\in(1,1+\xi).$$
 Furthermore, following the proof of $(i)$, since $E_a$ is nonincreasing, if $E_a<0$, for any $b\in(a,+\infty)$, we can get some uniform $\xi>0$ satisfying
	$$E_{\theta c}\le\theta E_{c},\ \forall\theta\in[1,1+\xi),\forall c\in[a,b].$$
Now, for any $a>0$ with $E_a<0$ and $\theta>1$, we take $\xi>0$ such that$$E_{(1+k)c}\le(1+k)E_{c},\forall k\in[0,\xi),\forall c\in[a,\theta b]$$
	Then, we may choose $k_0\in(0,\xi)$ and $n\in\mathbb{N}$ such that
	$$(1+k_0)^n<\theta<(1+k_0)^{n+1},$$
	and so
	\begin{align*}
		E_{\theta a}=E_{(1+k_0)\frac{\theta}{1+k_0}a}\le&(1+k_0)E_{\frac{\theta}{1+k_0}a}\le(1+k_0)^2E_{\frac{\theta}{(1+k_0)^2}a}\\
		\le&(1+k_0)^nE_{\frac{\theta}{(1+k_0)^n}a}\le(1+k_0)^n\frac{\theta}{(1+k_0)^n}E_{a}=\theta E_{a}.
	\end{align*}
	Then, if $E_a$ is attained, we get that $E_{\theta a}<\theta E_{a}$ for any $\theta>1$. For $0<b\le a$, we obtain that
$$
	E_{a+b}=E_{\frac{a+b}{a}a}\le\frac{a+b}{a}E_{a}=E_{a}+\frac{b}{a}E_{a}=E_a+\frac{b}{a}E_{\frac{a}{b}b}\le E_{a}+E_{b}.
$$
If $E_a$ or $E_b$ is attained, we get
\begin{equation}\lab{aab}
E_{a}=E_{\frac{a}{b}b}<\frac{a}{b}E_{b},
\end{equation}
and then $E_{a+b}<E_{a}+E_{b}.$
The proof is complete.
\end{proof}

The next compactness lemma on $S_a$ is useful in the study of the autonomous problem as well as non-autonomous problem.
\begin{lemma}\lab{2.4}
	Let $\{u_n\}\subset S_a$ be a minimizing sequence with respect to $E_{a}$. Then, for some subsequence, one of the following alternatives holds:
	\begin{itemize}
		\item [$(i)$] $\{u_n\}$ is strongly convergent;
		\item [$(ii)$]  There exists $\{y_n\}\subset S_a$ with $|y_n|\to\infty$ such that the sequence $v_n(x)=u_n(x+y_n)$ is strongly convergent to a function $v\in S_a$ with $J(v)=E_{a}$.
	\end{itemize}
\end{lemma}
\begin{proof}
	By Lemma \ref{1}, we know $J$ is coercive on $S_a$, the sequence $\{u_n\}$ is bounded, so $u_n\rightharpoonup u$ in $H^s(\R^N)$ for some subsequence. Now we consider the following three possibilities.
	
	$(1)$ If $u\not\equiv0$ and $|u|^2_2=b\neq a$, we must have $b\in(0,a)$. Set $v_n=u_n-u$, by the Br$\acute{e}$zis-Lieb Lemma \cite{Will},
	\begin{align}
		\iint_{\R^{2N}}\frac{|u_n(x)-u_n(y)|^2}{|x-y|^{N+2s}}\mathrm{d}x\mathrm{d}y=&\iint_{\R^{2N}}\frac{|v_n(x)-v_n(y)|^2}{|x-y|^{N+2s}}\mathrm{d}x\mathrm{d}y\nonumber \\
		&+\iint_{\R^{2N}}\frac{|u(x)-u(y)|^2}{|x-y|^{N+2s}}\mathrm{d}x\mathrm{d}y+o_n(1).
	\end{align}
	Since $F$ is a $C^1$ function and has a subcritical growth in the Sobolev sense, then it follows that
	\begin{equation}\lab{qqq2}
		\int_{\R^N}F(u_n)\mathrm{d}x=\int_{\R^{N}}F(u_n-u)\mathrm{d}x+\int_{\R^N}F(u)\mathrm{d}x+o_n(1).
	\end{equation}
	Furthermore, setting $d_n=|v_n|^2_2$, and by using
	$$|u_n|^2_2=|v_n|^2_2+|u_n|^2_2+o_n(1),$$
	we obtain that $d_n\in(0,a)$ for $n$ large enough and $|v_n|^2_2\to d$ with $a=b+d$, we infer that
	\begin{align*}
		E_{a}+o_n(1)=&J(u_n)\\
		=&\frac{1}{2}\iint_{\R^{2N}}\frac{|v_n(x)-v_n(y)|^2}{|x-y|^{N+2s}}\mathrm{d}x\mathrm{d}y+\frac{\eta}{2}|v_n|^2_2-\mu\int_{\R^N}F(v_n)\mathrm{d}x\\
		&+\frac{1}{2}\iint_{\R^{2N}}\frac{|u(x)-u(y)|^2}{|x-y|^{N+2s}}\mathrm{d}x\mathrm{d}y+\frac{\eta}{2}|u|^2_2-\mu\int_{\R^N}F(u)\mathrm{d}x+o_n(1)\\
		=&J(v_n)+J(u)+o_n(1)\\
		\ge&E_{d_n}+E_{b}+o_n(1).
	\end{align*}
	Letting $n\to+\infty$, by Lemma \ref{3}, we find that
	\begin{align*}
		E_{a}\ge&E_{d}+E_{b}>E_{a},
	\end{align*}
	which is a contradiction. This possibility can not exist.
	
	$(2)$ If $|u_n|^2_2=|u|^2_2=a$, it is well known that $u_n\to u$ in $L^2(\R^N)$. Then, by \eqref{2333} and \eqref{23}, we have that
	\begin{align*}
		\int_{\RN}F(u_n-u)\mathrm{d}x\le&C_1\int_{\RN}|u_n-u|^q\mathrm{d}x+C_2\int_{\RN}|u_n-u|^p\mathrm{d}x\\
		\le&C(\int_{\RN}|u_n-u|^2)^{\frac{q}{2}-\frac{N(q-2)}{4s}}+C(\int_{\RN}|u_n-u|^2)^{\frac{p}{2}-\frac{N(p-2)}{4s}}
	\end{align*}
	Hence, we get $\int_{\RN}F(u_n-u)\mathrm{d}x\rightarrow0$. From \eqref{qqq2}, we obtain that
	$$\int_{\R^N}F(u_n)\mathrm{d}x\to\int_{\R^N}F(u)\mathrm{d}x.$$
	which combines with $E_{a}=\lim\limits_{n\to+\infty}J(u_n)$ 
	provide
	\begin{align*}
		E_{a}=&\lim\limits_{n\to+\infty}\frac{1}{2}\int_{\R^N}|(-\Delta)^{\frac{s}{2}}u_n|^{2}+\eta u^2_n)\mathrm{d}x-\mu\int_{\R^N}F(u)\mathrm{d}x\\
		\ge&\frac{1}{2}\int_{\R^N}|(-\Delta)^{\frac{s}{2}}u|^{2}+\eta u^2)\mathrm{d}x-\mu\int_{\R^N}F(u)\mathrm{d}x=J(u)\\
		\ge&E_{a},
	\end{align*}
	Since $u\in S_a$, we infer that $E_{a}=J(u)$, then $||u_n||^2\to||u||^2$,
	where $||\  ||$ denotes the usual norm in $H^s(\R^N)$. Thus $u_n\to u$ in $H^s(\R^N)$, which implies that $(i)$ occurs.
	
	$(3)$ If $u\equiv0$, that is, $u_n\rightharpoonup0$ in $H^s(\R^N)$.
	We claim that there exists $\beta >0$  such that
	\begin{equation}\label{qqqq2}
		\liminf\limits_{n\rightarrow+\infty}\sup\limits_{y\in\RN}\int_{B_R(y)}|u_n|^2\mathrm{d}x\ge\beta,\  \ \mbox{for some}\  R>0.
	\end{equation}
	Indeed, otherwise by \cite[Lemma 2.2]{Felmer}, we have $u_n\to0$ in $L^l(\R^N)$ for all $l\in(2,\frac{2N}{N-2s})$. Thus
	\begin{align*}
		E_{a}+o_n(1)=J(u_n)=&\frac{1}{2}\int_{\R^N}|(-\Delta)^{\frac{s}{2}}u_n|^{2}\mathrm{d}x+\frac{\eta }{2}\int_{\RN}u^2_n\mathrm{d}x-\mu\int_{\R^N}F(u_n)\mathrm{d}x\\
		=&\frac{1}{2}\int_{\R^N}|(-\Delta)^{\frac{s}{2}}u_n|^{2}\mathrm{d}x+\frac{\eta }{2}\int_{\RN}u^2_n\mathrm{d}x+o_n(1)
	\end{align*}
	which contradicts the Lemma \ref{2}.
	
	Hence, from this case, \eqref{qqqq2} holds and $|y_n|\rightarrow+\infty$, then we consider $\tilde{u}_n(x)=u(x+y_n)$, obviously $\{\tilde{u}_n\}\subset S_a$ and it is also a minimizing sequence with respect to $J_{a}$. It is observed that there exists $\tilde{u}\in H^s(\R^N)\backslash\{0\}$ such that $\tilde{u}_n(x)\rightharpoonup\tilde{u}$ in $H^s(\R^N)$.
	Following as in the first two possibilities of the proof, we infer that $\tilde{u}_n(x)\to\tilde{u}$ in $H^s(\R^N)$, which implies that $(ii)$ occurs. This proves the lemma.
\end{proof}
In what follows, we begin to prove Theorem \ref{Th2.1}.
\begin{proof}[{\bf Proof of Theorem \ref{Th2.1}}]
	 By Lemma \ref{1}, Lemma \ref{2}, there exists a bounded minimizing sequence $\{u_n\}\subset S_a$ satisfying $J(u_n)\to E_{a}$. Then applying Lemma \ref{2.4}, there exists $u\in S_a$ such that $J(u)=E_{a}$. By the
	Lagrange multiplier, there exists $\lambda\in\R$ such that
	\begin{equation}\lab{qq2}
	J'(u)=\lambda \Phi'(u)\ \ \mbox{in}\ \ H^s(\RN)',
	\end{equation}
where $\Phi(u):H^s(\RN)\rightarrow \R$ is given by
$$\Phi(u)=\frac{1}{2}\int_{\RN}|u|^2\mathrm{d}x,\ \ \ u\in H^s(\RN).$$
Therefore, from \eqref{qq2}, we have
\begin{equation}\lab{ss}
	(-\Delta)^su+\eta u=\lambda u+\mu f(u)\ \ \mbox{in}\ \ \RN,
\end{equation}
By Lemma \ref{ee}, we can get
\begin{align*}
	(\lambda-\eta)\int_{\RN}u^2\mathrm{d}x=&\int_{\R^N}|(-\Delta)^{\frac{s}{2}}u|^{2}\mathrm{d}x-\mu\int_{\RN}f(u)u\mathrm{d}x\\
	=&-\frac{N\mu}{s}\int_{\RN}F(u)\mathrm{d}x+\frac{N\mu}{2s}\int_{\RN}f(u)u\mathrm{d}x-\mu\int_{\RN}f(u)u\mathrm{d}x\\
	=&-\frac{\mu}{s}[\int_{\RN}NF(u)-\frac{N-2s}{2}f(u)u\mathrm{d}x].
\end{align*}
Furthermore, according to the condition $(f_3)$ and the claim $3$, we must have $\lambda<\eta$.

Next, we will prove that $u$ can be chosen to be positive. Obviously, we  have $J(u)=J(|u|)$. Moreover, since $u\in S_a$ shows that $|u|\in S_a$, we infer that $$E_{a}=J(u)=J(|u|)\ge E_{a}.$$
	which implies that $J(|u|)=E_{a}$, we can replace $u$ by $|u|$. Furthermore, if $u^*$ denotes the Symmetrization radial decreasing rearrangement of $u$ (see \cite[Section 9]{Almgren}), we observe that
	\begin{align}
		\iint_{\R^{2N}}\frac{|u^*(x)-u^*(y)|^2}{|x-y|^{N+2s}}\mathrm{d}x\mathrm{d}y\le\iint_{\R^{2N}}\frac{|u(x)-u(y)|^2}{|x-y|^{N+2s}}\mathrm{d}x\mathrm{d}y\nonumber\\
		\int_{\RN}|u|^2\mathrm{d}x=\int_{\RN}|u^*|^2\mathrm{d}x \ \mbox{and} \ \int_{\RN}F(u)\mathrm{d}x=\int_{\RN}F(u^*)\mathrm{d}x
	\end{align}
	then $u^*\in S_a$ and $J(u^*)=E_{a}$, it follows that we can replace $u$ by $u^*$. Similarly as in \cite{Liu-Zhang}, one can show that $u(x)>0$ for any $x\in\R$. This completes the proof.
	

\end{proof}


\vskip0.1in


\s{The non-autonomous problem}
\renewcommand{\theequation}{3.\arabic{equation}}

In this section, we first give some properties of the functional $I_\epsilon(u)$ given  by \eqref{xx4} restricted to the sphere $S_a$, and then prove Theorem \ref{Th1}. Define the following energy functionals
$$
I_{\infty}(u)=\frac{1}{2}\int_{\R^N}|(-\Delta)^{\frac{s}{2}}u|^{2}\mathrm{d}x-h_{\infty}\int_{\R^N}F(u)\mathrm{d}x
$$
and for $i=1,2,\cdots, k$,
\begin{align*}
	I_{a_i}(u)=&\frac{1}{2}\int_{\R^N}|(-\Delta)^{\frac{s}{2}}u|^{2}\mathrm{d}x+\frac{V(a_i)}{2}\int_{\RN}u^2\mathrm{d}x-h(a_i)\int_{\R^N}F(u)\mathrm{d}x.
\end{align*}
Moreover, denoted by $E_{\epsilon,a}$, $E_{a_i,a}$ and $E_{\infty,a}$ the following real numbers
$$E_{\epsilon,a}=\inf\limits_{u\in S_a}I_\epsilon(u),\ \ E_{a_i,a}=\inf\limits_{u\in S_a}I_{a_i}(u),\ \ E_{\infty,a}=\inf\limits_{u\in S_a}I_\infty(u).$$

The next two lemmas establish some crucial relations involving the levels $E_{\epsilon,a}$, $E_{\infty,a}$ and $E_{a_i,a}$. 
For any $\alpha,\beta\in\R$, set
$$J_{\alpha\beta}(u)=\frac{1}{2}\int_{\R^N}|(-\Delta)^{\frac{s}{2}}u|^{2}\mathrm{d}x+\frac{\beta}{2}\int_{\RN}u^2\mathrm{d}x-\alpha\int_{\R^N}F(u)\mathrm{d}x=E_{h_1V_1,a}.$$
	where
	$$E_{\alpha\beta,a}=\inf\limits_{u\in S_a}J_{\alpha\beta}(u),$$
\begin{lemma}\lab{q3}
	Fix $a>0$, let $0<h_1<h_2$ and $V_2<V_1\leq0$. Then $E_{h_2V_2,a}<E_{h_1V_1,a}<0$.
\end{lemma}
\begin{proof}
The proof is standard and we omit the details.
\end{proof}

\begin{lemma}\lab{t3}
	$\limsup\limits_{\epsilon\rightarrow0^+}E_{\epsilon,a}\le E_{a_i,a}<E_{\infty,a}<0, i=1,2,\cdots,k$.
\end{lemma}
\begin{proof}
	By the proof of the Theorem \ref{Th2.1}, choose $u_0\in S_a$ such that $I_{a_i}(u_0)=E_{a_i,a}$. For $1\le i\le k$, we define $$u=u_0(x-\frac{a_i}{\epsilon}),\ x\in\RN.$$
Then $u\in S_a$ for all $\epsilon>0$, we have
	\begin{align*}
		E_{\epsilon,a}\le I_\epsilon(u)=\frac{1}{2}|(-\Delta)^{\frac{s}{2}}u_0|^{2}_2+\frac{1}{2}\int_{\RN}V(\epsilon x+a_i)u^2_0\mathrm{d}x-\int_{\R^N}h(\epsilon x+a_i)F(u_0)\mathrm{d}x.
	\end{align*}
	Letting $\epsilon\rightarrow0^+$, by the Lebesgue dominated convergence theorem, we deduce
	\begin{equation}\lab{zz}
		\limsup\limits_{\epsilon\rightarrow0^+}E_{\epsilon,a}\le\lim\limits_{\epsilon\rightarrow0^+}I_{\epsilon}(u)=I_{a_i}(u_0)=E_{a_i,a}.
	\end{equation}
Noting that $E_{\infty,a}$ can be achieved, due to $0<h_{\iy}<h(a_i)$ and $V(a_i)<0$, we have
 $$E_{a_i,a}<E_{\infty,a}<0.$$
It completes the proof.
\end{proof}

Hence by Lemma \ref{t3}, there exists $\epsilon_{1}>0$ satisfying $E_{\epsilon,a}<E_{\infty,a}$ for all $\epsilon\in(0,\epsilon_{1})$, In the following, we always assume that $\epsilon\in(0,\epsilon_{1})$. The next three lemmas will be used to prove the $(PS)_c$ condition for $I_{\epsilon}$ restricts to $S_a$ at some levels.
\begin{lemma}\label{L}
	Assume $\{u_n\}\subset S_a$ such that $I_{\epsilon}(u_n)\rightarrow c$ as $n\rightarrow+\infty$ with $c<E_{\infty,a}<0$, then
	$$\delta:=\liminf\limits_{n\rightarrow\infty}\sup\limits_{y\in\RN}\int_{B(y,1)}|u_n(x)|^2\mathrm{d}x>0.$$
\end{lemma}
\begin{proof}
	We argue by contradiction and assume that $\delta=0$, then up to a subsequence, we have $u_n\rightarrow0$ in $L^l(\RN)$ for all $l\in(2,\frac{2N}{N-2s})$, by the Lebesgue dominated convergence theorem and $(f_1)$-$(f_2)$, we infer that
	\begin{equation}\lab{tt3}
		\int_{\RN} h(\epsilon x)F(u_n)\mathrm{d}x\rightarrow0\ \ \mbox{as}\ \ n\rightarrow+\infty.
	\end{equation}
Since $V(x)\rightarrow0$ as $|x|\rightarrow\infty$, one can show that
$$\int_{\RN}V(x)u_n^2dx=o_n(1),$$
which combining with \eqref{tt3} follows that
	\begin{align*}
		0>c=I_{\epsilon}(u_n)+o(1)=\frac{1}{2}\int_{\R^N}|(-\Delta)^{\frac{s}{2}}u_n|^{2}\mathrm{d}x+o(1)\ge 0,
	\end{align*}
which is a contradiction.
\end{proof}
\begin{lemma}\lab{ll3}
	Under the assumption of Lemma \ref{L}, assume $u_n\rightharpoonup u$ in $H^s(\RN)$, then $u\not\equiv0$.
\end{lemma}
\begin{proof}
	By Lemma \ref{L}, we have that
	$$\liminf\limits_{n\rightarrow\infty}\sup\limits_{y\in\RN}\int_{B_r(y)}|u_n(x)|^2\mathrm{d}x>0.$$
	So if $u\equiv0$, there exists $\{y_n\}$ satisfying $|y_n|\rightarrow\infty$, let $\tilde{u}_n=u_n(x+y_n)$, obviously $\{\tilde{u}_n\}\subset S_a$, we have
	\begin{align*}
		c+o_n(1)=&I_{\epsilon}(u_n)\\
		=&\frac{1}{2}\int_{\R^N}|(-\Delta)^{\frac{s}{2}}u_n|^{2}\mathrm{d}x+\frac{1}{2}\int_{\RN}V(\epsilon x)u^2_n\mathrm{d}x-\int_{\R^N}h(\epsilon x)F(u_n)\mathrm{d}x\\
		=&\frac{1}{2}\int_{\R^N}|(-\Delta)^{\frac{s}{2}}\tilde{u}_n|^{2}\mathrm{d}x+\frac{1}{2}\int_{\RN}V(\epsilon x+\epsilon y_n)\tilde{u}^2_n\mathrm{d}x-\int_{\R^N}h(\epsilon x+\epsilon y_n)F(\tilde{u}_n)\mathrm{d}x\\
		=&I_{\infty}(\tilde{u}_n)+\frac{1}{2}\int_{\RN}(V(\epsilon x+\epsilon y_n)-V_{\infty})\tilde{u}^2_n\mathrm{d}x+\int_{\RN}(h_{\infty}-h(\epsilon x+\epsilon y_n))F(\tilde{u}_n)\mathrm{d}x\\
		=&I_{\infty}(\tilde{u}_n)+o_n(1)\ge E_{\infty,a}+o_n(1),
	\end{align*}
	which is absurd, because $c<E_{\infty,a}<0$. This proves the lemma.
\end{proof}

\begin{lemma}\lab{lll3}
	Let $\{u_n\}\subset S_a$ be a $(PS)_c$ sequence of $I_{\epsilon}$ restricted to $S_a$ with $c<E_{\infty,a}<0$ and let $u_n\rightharpoonup u_{\epsilon}$ in $H^s(\RN)$. If $u_n\not\rightarrow u_{\epsilon}$ in $H^s(\RN)$, there exists $\beta>0$ independent of $\epsilon\in(0,\epsilon_{1})$ such that $$\liminf\limits_{n\rightarrow+\infty}|u_n-u_{\epsilon}|^2_2\ge\beta.$$
\end{lemma}
\begin{proof}
	Setting the functional $\Phi:H^s(\RN)\rightarrow\R$ given by $$\Phi(u)=\frac{1}{2}\int_{\R^N}|u|^2\mathrm{d}x,$$It follows that $S_a=\Phi^{-1}(\{a/2\})$. Then, by Willem \cite[Proposition 5.12]{Will}, there exists $\{\lambda_n\}\subset\R$ such that
	\begin{equation}\lab{h3}
		||I_{\epsilon}'(u_n)-\lambda_n\Phi'(u_n)||_{(H^s(\RN))'}\rightarrow0\ \ \mbox{as}\ n\rightarrow+\infty.
	\end{equation}
	By the boundedness of $\{u_n\}$ in $H^s(\RN)$, we know $\{\lambda_n\}$ is a bounded sequence, thus there exists $\lambda_{\epsilon}$ such that $\lambda_n\rightarrow\lambda_{\epsilon}$ as $n\rightarrow+\infty$. Then, together with \eqref{h3}, we get
	$$I_{\epsilon}'(u_{\epsilon})-\lambda_{\epsilon}\Phi'(u_{\epsilon})=0 \ \ \ \mbox{in}\ \  {(H^s(\RN))'},$$
	and setting $v_n=u_n-u_{\epsilon}$, we deduce that
	\begin{equation}\lab{hh3}
		||I_{\epsilon}'(v_n)-\lambda_n\Phi'(v_n)||_{(H^s(\RN))'}\rightarrow0\ \ \mbox{as}\ n\rightarrow+\infty.
	\end{equation}
	By a straightforward calculation, we have
	\begin{align*}
		E_{\infty,a}>&\liminf\limits_{n\rightarrow+\infty}I_{\epsilon}(u_n) \nonumber\\
		=&\liminf\limits_{n\rightarrow+\infty}(I_{\epsilon}(u_n)-\frac{1}{2}I_{\epsilon}'(u_n)u_n+\frac{1}{2}\lambda_na+o_n(1))\nonumber\\
		=&\liminf\limits_{n\rightarrow+\infty}[\int_{\RN}\frac{h(\epsilon x)}{2}f(u_n)u_n\mathrm{d}x-\int_{\RN}h(\epsilon x)F(u_n)\mathrm{d}x+\frac{1}{2}\lambda_na+o(1)]\nonumber\\
		\ge&\frac{1}{2}\lambda_{\epsilon}a
	\end{align*}
implying that
	\begin{equation}\lab{hhh3}
		\lambda_{\epsilon}\le\frac{2E_{\infty,a}}{a}<0,\ \mbox{for} \ \mbox{all}\ \epsilon\in(0,\epsilon_1).
	\end{equation}
From \eqref{hh3}, we get
	\begin{equation}\label{hhh}
		|(-\Delta)^{\frac{s}{2}}v_n|^{2}_2+\int_{\RN}V(\epsilon x)|v_n|^2\mathrm{d}x-\lambda_{\epsilon}|v_n|^2_2-\int_{\R^N}h(\epsilon x)f(v_n)v_n\mathrm{d}x=o_n(1).
	\end{equation}
which combined with \eqref{hhh3} to give
$$|(-\Delta)^{\frac{s}{2}}v_n|^{2}_2+\int_{\RN}V(\epsilon x)|v_n|^2\mathrm{d}x-\frac{2E_{\infty,a}}{a}\int_{\R^N}|v_n|^2\mathrm{d}x\le \int_{\R^N}h(\epsilon x)f(v_n)v_n\mathrm{d}x+o_n(1),$$
which leads to
\begin{equation}\lab{i3}
\int_{\R^N}|(-\Delta)^{\frac{s}{2}}v_n|^{2}\mathrm{d}x+C_3\int_{\R^N}|v_n|^2\mathrm{d}x\le C_2\int_{\RN}|v_n|^p\mathrm{d}x+o_n(1),
\end{equation}
for some constant $C_3>0$ that does not depend on $\epsilon\in(0,\epsilon_{1})$. If $u_n\not\rightarrow u_{\epsilon}$ in $H^s(\RN)$, that is $v_n\not\rightarrow 0$ in $H^s(\RN)$, we know that there exists $C_0>0$ independent of $\epsilon$ such that
	\begin{equation}\lab{ii3}
		\liminf\limits_{n\rightarrow+\infty}|v_n|^p_p\ge C_0,
	\end{equation}
Then, by the fractional Gagliardo-Nirenberg-sobolev inequality,
	$$\int_{\R^N}|v_n|^\alpha\le C(s,N,\alpha)(\int_{\R^N}|(-\Delta)^{\frac{s}{2}}v_n|^2)^{\frac{N(\alpha-2)}{4s}}(\int_{\R^N}|v_n|^2)^{\frac{\alpha}{2}-\frac{N(\alpha-2)}{4s}},$$
	for some positive constant $C(s,N,\alpha)>0$. We have
	\begin{align}\lab{zzz}
		\liminf\limits_{n\rightarrow+\infty}\int_{\RN}|v_n|^p\le& C(s,N,p)(\int_{\R^N}|(-\Delta)^{\frac{s}{2}}v_n|^2)^{\frac{N(p-2)}{4s}}(\liminf\limits_{n\rightarrow+\infty}\int_{\R^N}|v_n|^2)^{\frac{p}{2}-\frac{N(p-2)}{4s}}\nonumber\\
		\le&C(s,N,p)K^{\frac{N(p-2)}{4s}}(\liminf\limits_{n\rightarrow+\infty}\int_{\R^N}|v_n|^2)^{\frac{p}{2}-\frac{N(p-2)}{4s}}
	\end{align}
	Clearly also, for $K>0$ is a suitable constant independent of $\epsilon$ satisfying the condition $\int_{\R^N}|(-\Delta)^{\frac{s}{2}}v_n|^2\le K$. This together with \eqref{ii3} and \eqref{zzz} gives that there exists $\beta>0$ independent of $\epsilon\in(0,\epsilon_{1})$ such that $$\liminf\limits_{n\rightarrow+\infty}|v_n|^2_2\ge\beta.$$we get desired result.
\end{proof}

Next we will give the compactness lemma.
\begin{lemma}\lab{x4}
	Let$$o<\rho_0<min\{E_{\infty,a}-E_{a_i,a},\frac{\beta}{a}(E_{\infty,a}-E_{a_i,a})\}.$$
	Then, for each $\epsilon\in(0,\epsilon_{1})$, the functional $I_{\epsilon}$ satisfies the $(PS)_c$ condition restricts to $S_a$ if $c<E_{a_i,a}+\rho_0$.
\end{lemma}
\begin{proof}
	Let $\{u_n\}$ be a $(PS)_c$ sequence for $I_{\epsilon}$ restricts to
	$S_a$ and $c<E_{a_i,a}+\rho_0$. It follows that $c<E_{\infty,a}<0$, since $\{u_n\}$ is bounded in $H^s(\RN)$, let $u_n\rightharpoonup u_{\epsilon}$ in $H^s(\RN)$. By Lemma \ref{ll3}, $u_\epsilon\not\equiv0$. Denote $v_n=u_n-u_{\epsilon}$, If $u_n\rightarrow u_{\epsilon}$ in $H^s(\RN)$, the proof is complete. If $u_n\not\rightarrow u_{\epsilon}$ in $H^s(\RN)$, by Lemma \ref{lll3},
	$$\liminf\limits_{n\rightarrow+\infty}|v_n|^2_2\ge\beta.$$
	Set $b=|u_{\epsilon}|^2_2$, $d_n=|v_n|^2_2$ and suppose that $|v_n|^2_2\rightarrow d>0$, then we get $d\ge \beta>0$ and $a=b+d$. From $d_n\in(0,a)$ for $n$ large enough, we get
	\begin{equation}\lab{j3}
		c+o_n(1)=I_{\epsilon}(u_n)=I_{\epsilon}(v_n)+I_{\epsilon}(u_{\epsilon})+o_n(1).
	\end{equation}
	since $v_n\rightharpoonup0$ in $H^s(\RN)$, we can follow the lines in the proof of Lemma \ref{ll3}. Then
	\begin{equation}
		I_{\epsilon}(v_n)\ge E_{\infty,d_n}+o_n(1),
	\end{equation}
	which combing with \eqref{j3}, we obtain that
	\begin{align*}
		c+o_n(1)=I_{\epsilon}(u_n)\ge&E_{\infty,d_n}+I_{\epsilon}(u_{\epsilon})+o_n(1)\\
		\ge&E_{\infty,d_n}+E_{a_i,b}+o_n(1),
	\end{align*}
	Letting $n\rightarrow\infty$, by the inequation \eqref{aab}, we have
	\begin{align*}
		c\ge E_{\infty,d}+E_{a_i,b}\ge&\frac{d}{a}E_{\infty,a}+\frac{b}{a}E_{a_i,a}\nonumber\\
		=&E_{a_i,a}+\frac{d}{a}(E_{\infty,a}-E_{a_i,a})\nonumber\\
		\ge&E_{a_i,a}+\frac{\beta}{a}(E_{\infty,a}-E_{a_i,a})
	\end{align*}
	which is a contradiction, because $c<E_{a_i,a}+\frac{\beta}{a}(E_{\infty,a}-E_{a_i,a})$. Therefore, we can obtain $u_n\rightarrow u_{\epsilon}$ in $H^s(\RN).$
\end{proof}

In what follows, let us fix $\bar{\rho},\bar{r}>0$ satisfying:
\begin{itemize}
	\item [(1)] $\overline{B_{\bar{\rho}}(a_i)}\cap\overline{B_{\bar{\rho}}(a_j)}$ for $i\not=j$ and $i,j\in\{1,\dots k\}$.
	\item [(2)] $\cup_{i=1}^kB_{\bar{\rho}}(a_i)\subset B_{\bar{r}}(0)$.
	\item [(3)] $Q_{\frac{\bar{\rho}}{2}}=\cup_{i=1}^l\overline{B_{\frac{\bar{\rho}}{2}}(a_i)}$.
\end{itemize}

We set the function $G_{\epsilon}:H^s(\RN)\backslash\{0\}\rightarrow\RN$ by
$$G_{\epsilon}(u)=\frac{\int_{\RN}\chi(\epsilon x)|u|^2\mathrm{d}x}{\int_{\RN}|u|^2\mathrm{d}x},$$
where $\chi:\RN\rightarrow\RN$ denotes the characteristic function, that is,
\begin{eqnarray*}
	\chi(x)=\left\{
	\begin{array}{l}
		x,\ \ \mbox{if}\ |x|\le\bar{r},\\
		\bar{r}\frac{x}{|x|},\ \ \mbox{if}\ |x|>\bar{r}.
	\end{array}
	\right.
\end{eqnarray*}

The next two lemmas will be useful to get important $(PS)$ sequences for $I_{\epsilon}$ restricted to $S_a$.
\begin{lemma}\lab{kkk3}
	For $\epsilon\in(0,\epsilon_{1})$, there exist $\delta_1>0$ such that if $u\in S_a$ and $I_{\epsilon}(u)\le E_{a_i,a}+\delta_1$, then $$G_{\epsilon}(u)\in Q_{\frac{\bar{\rho}}{2}},\forall\epsilon\in(0,\epsilon_{1}).$$
\end{lemma}
\begin{proof}
	If the lemma does not occur, there must be $\delta_{n}\rightarrow0$, $\epsilon_{n}\rightarrow0$ and $\{u_n\}\subset S_a$ such that
	\begin{equation}\lab{k3}
		I_{\epsilon_n}(u_n)\le E_{a_i,a}+\delta_{n}\  \mbox{and}\  G_{\epsilon_n}(u_n)\not\in Q_{\frac{\bar{\rho}}{2}},\forall\epsilon\in(0,\epsilon_{1}).
	\end{equation}
	so we have$$E_{a_i,a}\le I_{a_i}(u_n)\le I_{\epsilon_{n}}(u_n)\le E_{a_i,a}+\delta_{n}$$then$$\{u_n\}\subset S_a\  \mbox{and}\  I_{a_i}(u_n)\rightarrow E_{a_i,a}.$$
	According to Lemma \ref{2.4}, we have one of the following two cases:
	\begin{itemize}
		\item [$(i)$] ${u_n}\rightarrow u$ in $H^s(\RN)$ for some $u\in S_a$,
		\item [$(ii)$]  There exists $\{y_n\}\subset S_a$ with $|y_n|\to\infty$ such that the sequence $v_n(x)=u_n(x+y_n)$ in $H^s(\RN)$ to some $v\in S_a$.
	\end{itemize}
	
	For $(i)$: By Lebesgue dominated convergence theorem,$$G_{\epsilon_n}(u_n)=\frac{\int_{\RN}\chi(\epsilon x)|u_n|^2\mathrm{d}x}{\int_{\RN}|u_n|^2\mathrm{d}x}\rightarrow\frac{\int_{\RN}\chi(0)|u|^2\mathrm{d}x}{\int_{\RN}|u|^2\mathrm{d}x}=0\in Q_{\frac{\bar{\rho}}{2}}.$$
	Then $G_{\epsilon_n}(u_n)\in Q_{\frac{\bar{\rho}}{2}}$ for $n$ large enough, that contradicts \eqref{k3}.
	
	For $(ii)$: We will study the following two case: $\rm (I)$ $|\epsilon_{n}y_n|\rightarrow+\infty$; $\rm (II)$ $\epsilon_{n}y_n\rightarrow y$ for some $y\in\RN$.
	
	If $\rm (I)$ holds, the limit $v_n\rightarrow v$ in $H^s(\RN)$ provides
	\begin{align}\lab{y3}
		I_{\epsilon_{n}}(u_n)=&\frac{1}{2}|(-\Delta)^{\frac{s}{2}}v_n|^{2}_2+\frac{1}{2}\int_{\RN}V(\epsilon_n x+\epsilon_n y_n)|v_n|^2\mathrm{d}x-\int_{\R^N}h(\epsilon_{n}x+\epsilon_{n}y_n)F(v_n)\mathrm{d}x\nonumber\\
		\rightarrow& I_{\infty}(v)\ \mbox{as}\ n\rightarrow+\infty.
	\end{align}
	Since $I_{\epsilon}(u_n)\le E_{a_i,a}+\delta_{n}$, we deduce that $$E_{\infty,a}\le I_{\infty}(v)\le E_{a_i,a}.$$
	which contradicts $E_{a_i,a}<E_{\infty,a}$ in Lemma \ref{t3}.
	
	If $\rm (II)$ holds, by $\eqref{y3}$, we obtain that
	$$I_{\epsilon_{n}}(u_n)\rightarrow I_{h(y)V(y)}(v) \ \ \mbox{as}\ \ n\rightarrow+\infty,$$
	and then $E_{h(y)V(y),a}\le I_{h(y)V(y)}(v)\le E_{a_i,a}$. By Lemma \ref{q3}, we must have $h(y)=h(a_i)$ and $V(y)=V(a_i)$. Namely, $y=a_i$ for some $i=1,2,\dots,k$. Hence
	\begin{align*}
		G_{\epsilon_n}(u_n)=\frac{\int_{\RN}\chi(\epsilon_{n} x)|u_n|^2\mathrm{d}x}{\int_{\RN}|u_n|^2\mathrm{d}x}
		=&\frac{\int_{\RN}\chi(\epsilon_{n}x+\epsilon_{n}y_n)|v_n|^2\mathrm{d}x}{\int_{\RN}|v_n|^2\mathrm{d}x}\\
		\rightarrow&\frac{\int_{\RN}\chi(y)|v|^2\mathrm{d}x}{\int_{\RN}|v|^2\mathrm{d}x}=0\in Q_{\frac{\bar{\rho}}{2}}
	\end{align*}
	which implies that $G_{\epsilon_n}(u_n)\in Q_{\frac{\bar{\rho}}{2}}$ for $n$ large enough, That contradicts \eqref{k3}. The proof is complete.
\end{proof}

From now on, we will use the following notations:
\begin{itemize}
	\item [$\bullet$] $\theta_{\epsilon}^i$:=$\{u\in S_a:|G_{\epsilon}(u)-a_i|\le\bar{\rho}\}$;
	\item [$\bullet$] $\partial \theta_{\epsilon}^i$:=$\{u\in S_a:|G_{\epsilon}(u)-a_i|=\bar{\rho}\}$;
	\item [$\bullet$] $\beta_{\epsilon}^i$=$\inf\limits_{u\in\theta_{\epsilon}^i}I_{\epsilon}(u)$;
	\item [$\bullet$] $\bar{\beta}_{\epsilon}^i$=$\inf\limits_{u\in\partial \theta_{\epsilon}^i}I_{\epsilon}(u)$.
\end{itemize}
\begin{lemma}\lab{q4}
	Let  $\rho_0$ be defined in lemma \ref{x4}. Then there is  $$\beta_{\epsilon}^i<E_{a_i,a}+\rho_0\ \mbox{and} \ \beta_{\epsilon}^i<\bar{\beta}_{\epsilon}^i,\  for\ \ \forall \epsilon\in(0,\epsilon_{1}).$$
\end{lemma}
\begin{proof}
	Let $u\in S_a$ satisfy $$I_{a_i}(u)=E_{a_i,a}.$$
	For $1\le i\le k$, we define $$\hat{u}_{\epsilon}^i(x):=u(x-\frac{a_i}{\epsilon}),\ x\in\RN.$$
	Then $\hat{u}_{\epsilon}^i(x)\in S_a$ for all $\epsilon>0$, by direct calculations give that$$I_{\epsilon}(\hat{u}_{\epsilon}^i(x))=\frac{1}{2}|(-\Delta)^{\frac{s}{2}}u|^{2}_2+\frac{1}{2}\int_{\RN}V(\epsilon x+a_i)|u|^2\mathrm{d}x-\int_{\R^N}h(\epsilon x+a_i)F(u)\mathrm{d}x,$$and then
	\begin{equation}\lab{kk3}
		\lim\limits_{\epsilon\rightarrow0}I_{\epsilon}(\hat{u}_{\epsilon}^i)=I_{a_i}(u)=E_{a_i,a}.
	\end{equation}
	we know $$G_{\epsilon}(\hat{u}_{\epsilon}^i)=\frac{\int_{\RN}\chi(\epsilon x+a_i)|u|^2\mathrm{d}x}{\int_{\RN}|u|^2\mathrm{d}x}\rightarrow a_i \ \ \mbox{as}\ \  \epsilon\rightarrow0^+.$$
	so $\hat{u}_{\epsilon}^i(x)\in\theta_{\epsilon}^i$ for $\epsilon$ small enough, which combined with \eqref{kk3} implies that
	$$I_{\epsilon}(\hat{u}_{\epsilon}^i)<E_{a_i,a}+\frac{\delta_1}{2},\ \forall \epsilon\in(0,\epsilon_{1}).$$
	Decreasing $\delta_1$ if necessary, we know that
	$$\beta_{\epsilon}^i<E_{a_i,a}+\rho_0,\ \forall \epsilon\in(0,\epsilon_{1}).$$
	For any $u\in\partial\theta_{\epsilon}^i$, that is $u\in S_a$ and $|G_{\epsilon}(u)-a_i|=\bar{\rho}$, we get that $|G_{\epsilon}(u)|\not\in Q_{\frac{\bar{\rho}}{2}}$.
	Then by Lemma \ref{kkk3},
	$$I_{\epsilon}(u)>E_{a_i,a}+\delta_1,\ \ \mbox {for}\  \mbox{all} \ u\in\partial\theta_{\epsilon}^i\  \mbox{and} \ \epsilon\in(0,\epsilon_{1})$$
	which implies that $$\bar{\beta}_{\epsilon}^i=\inf\limits_{u\in\partial \theta_{\epsilon}^i}I_{\epsilon}(u)\ge E_{a_i,a}+\delta_1,$$
	Then, we have  $$\beta_{\epsilon}^i<\bar{\beta}_{\epsilon}^i,\ \ \mbox {for} \ \mbox{all} \  \epsilon\in(0.\epsilon_1).$$
\end{proof}
\vskip0.1in

\s{Proof of Theorem \ref{Th1}}
\renewcommand{\theequation}{4.\arabic{equation}}

\begin{proof}
	By Lemma \ref{q4}, for each $i\in\{1,2,\dots,k\}$, we can use the Ekeland's variational principle to find a sequence $\{u_n^i\}\subset S_a$ satisfying
	$$I_{\epsilon}(u_n^i)\rightarrow\beta_{\epsilon}^i\ \ \mbox{and}\ \
	I_{\epsilon}(w)\ge I_{\epsilon}(u_n^i)-\frac{1}{n}||w-u_n^i||,\ \forall w\in
	\theta_{\epsilon}^i,$$
	Recalling Lemma \ref{q4}, $\beta_{\epsilon}^i<\bar{\beta}_{\epsilon}^i$ , and so $u_n^i\in\theta_{\epsilon}^i\backslash\pa\theta_{\epsilon}^i$ for $n$ large enough.
	
	Let $w\in T_{u_n^i}S_a$, there exists $\delta>0$ such that the path $\gamma:(-\delta,\delta)\rightarrow S_a$ defined by $$\gamma(t)=a\frac{(u_n^i+tw)}{|u_n^i+tw|_2},$$
	and satisfies $$\gamma(t)\in\theta_{\epsilon}^i\backslash\pa\theta_{\epsilon}^i\ \ \forall t\in(-\delta,\delta),\ \gamma(0)=u_n^i\ \mbox{and}\ \gamma'(0)=w.$$Then for any $t\in(0,\delta)$,
	$$\frac{I_{\epsilon}(\gamma(t))-I_{\epsilon}(\gamma(0))}{t}=\frac{I_{\epsilon}(\gamma(t))-I_{\epsilon}(u_n^i)}{t}\ge-\frac{1}{n}||\frac{\gamma(t)-u_n^i}{t}||=-\frac{1}{n}||\frac{\gamma(t)-\gamma(0)}{t}||,$$
	Taking the limit of $t\rightarrow0^+$, we get $I'_{\epsilon}(u_n^i)w\ge-\frac{1}{n}||w||$. Replacing $w$ by $-w$, we obtain $|I'_{\epsilon}(u_n^i)w|\le\frac{1}{n}||w||$.
	Then, we have
	$$\mbox{sup}\{|I'_{\epsilon}(u_n^i)(w)|:||w||\le\delta_n\}\le\frac{1}{n},$$
	Consequently,
	$$I_{\epsilon}(u_n^i)\rightarrow\beta_{\epsilon}^i\ \ \mbox{as}\ n\rightarrow+\infty\ \ \mbox{and}\ \ ||I_{\epsilon}|'_{S_a}(u_n^i)||\rightarrow0\ \mbox{as}\ n\rightarrow+\infty,$$
	that is, $\{u_n^i\}$ is a $(PS)_{\beta_{\epsilon}^i}$ for $I_{\epsilon}$ restricts to $S_a$. Since $\beta_{\epsilon}^i<E_{a_i,a}+\rho_0$, it follows from Lemma \ref{x4}, there exists $u^i$ such that $u_n^i\rightarrow u^i$ in $H^s(\RN)$. Then, we get
	$$u^i\in\theta_{\epsilon}^i, I_{\epsilon}(u_n^i)=\beta_{\epsilon}^i\ \  \mbox{and}\ \  I_{\epsilon}|'_{S_a}(u_n^i)=0.$$
	Morever
	$$G_{\epsilon}(u^i)\in\overline{B_{\bar{\rho}}(a_i)}, \ \ G_{\epsilon}(u^j)\in\overline{B_{\bar{\rho}}(a_j)}$$and $$\overline{B_{\bar{\rho}}(a_i)}\cap\overline{B_{\bar{\rho}}(a_j)}=\emptyset \ \mbox{for}\ \ i\not=j,$$
	which implies that $u^i\not=u^j$ for $i\not=j$ while $1\le i,j\le k$, we can get $I_{\epsilon}$ has at least $k$ nontrivial critical points for any $\epsilon\in(0,\epsilon_1)$. Therefore, we obtain the theorem.
\end{proof}

\vskip 5mm
\noindent\textbf{Conflict of interest.} The authors have no competing interests to declare for this article.

\smallskip
\noindent\textbf{Data availability statement.} We declare that the manuscript has no associated data.

\medskip


\end{document}